\documentclass[12pt,a4paper]{amsart}

\usepackage{mathrsfs,amssymb}

\usepackage{color}

\newtheorem{theorem}{Theorem}[section]
\newtheorem*{theoremA}{Theorem A}
\newtheorem{lemma}[theorem]{Lemma}
\newtheorem{prop}[theorem]{Proposition}
\newtheorem{cor}[theorem]{Corollary}

\theoremstyle{definition}

\theoremstyle{remark}
\newtheorem{remark}[theorem]{Remark}

\numberwithin{equation}{section}

\let \la=\lambda
\let \e=\varepsilon
\let \d=\delta
\let \o=\omega
\let \a=\alpha

\let \g=\gamma

\let \si=\sigma

\newcommand{\W}{\mathcal{W}}

\begin{document}

\title[On weighted norm inequalities]
{On weighted norm inequalities for the Carleson and Walsh-Carleson operators}

\author{Francesco Di Plinio}
\address{Dipartimento di Matematica,    Universit\`a   di Roma ``Tor Vergata'' \newline  \indent Via della Ricerca Scientifica,   00133 Roma,  Italy \newline \indent   \centerline{and}   \indent
  The Institute for Scientific Computing and Applied Mathematics, \newline \indent
 Indiana University
\newline\indent
831 East Third Street, Bloomington, Indiana  47405, U.S.A. }
\email{diplinio@mat.uniroma2.it {\rm (F.\ Di Plinio)} }

\author{Andrei K. Lerner}
\address{Department of Mathematics, Bar-Ilan University,\newline \indent
 5290002 Ramat Gan, Israel}
	\email{aklerner@netvision.net.il {\rm (A.\ K.\ Lerner)}}

\thanks{The first  author is an INdAM - Cofund Marie Curie Fellow and is  partially
supported by the National Science Foundation under the grant
   NSF-DMS-1206438, and by the Research Fund of Indiana University. The second author was supported by the Israel Science Foundation (grant No. 953/13).}

\begin{abstract}
We prove $L^p(w)$ bounds for the Carleson operator ${\mathcal C}$, its lacunary version $\mathcal C_{lac}$, and its analogue for the Walsh series $\W$ in terms
of the $A_q$ constants $[w]_{A_q}$ for $1\le q\le p$. In particular,
we show that, exactly as for the Hilbert transform, $\|{\mathcal C}\|_{L^p(w)}$ is
bounded linearly by $[w]_{A_q}$ for $1\le q<p$.
 We also obtain $L^p(w)$ bounds in terms of $[w]_{A_p}$, whose sharpness is related to  certain conjectures (for instance, of Konyagin \cite{K2}) on pointwise convergence of Fourier series for functions near $L^1$.

Our approach works in the general
context of maximally modulated Calder\'on-Zygmund operators.
\end{abstract}



\maketitle

\section{Introduction}
For $f\in L^p({\mathbb R}), 1<p<\infty$, define the Carleson operator ${\mathcal C}$ by
\begin{equation}{\mathcal C}(f)(x)=\sup_{\xi\in {\mathbb R}}|H({\mathcal M}^{\xi}f)(x)|,\label{Carl}\end{equation}
where $H$ is the Hilbert transform, and ${\mathcal M}^{\xi}f(x)={\rm e}^{2\pi i\xi x}f(x)$.

The  celebrated Carleson-Hunt theorem on a.e.\ convergence of Fourier series in one of its equivalent
statements says that ${\mathcal C}$ is bounded on $L^p$ for any $1<p<\infty$. The crucial step was done
by Carleson \cite{C} who established that ${\mathcal C}$ maps $L^2$ into weak-$L^2$. After that Hunt \cite{H}
extended this result to any $1<p<\infty$. Alternative proofs of this theorem were obtained by Fefferman \cite{F} and by
Lacey-Thiele \cite{LT}. We refer also to \cite{Ar}, \cite[Ch.\ 11]{G} and \cite[Ch.\ 7]{MS}.

By a weight we mean a non-negative locally integrable function.
The weighted boundedness of ${\mathcal C}$   is also well known. Hunt-Young \cite{H} showed that ${\mathcal C}$ is bounded
on $L^p(w), 1<p<\infty,$ if $w$ satisfies the $A_p$ condition (see also \cite[p. 475]{G}). In \cite{GMS},
Grafakos-Martell-Soria extended this result to a more general class of maximally modulated singular integrals.
A variation  norm strengthening of the Hunt-Young result   was recently obtained by
Do-Lacey~\cite{DL}, relying on an adaptation of the phase plane analysis of \cite{LT} to the weighted setting.

In the past decade a great deal of attention was devoted to sharp $L^p(w)$ estimates for singular integral operators in terms of the $A_p$ constants $[w]_{A_p}$.
Recall that these constants are defined as follows:
$$[w]_{A_p}=\sup_{Q}\left(\frac{1}{|Q|}\int_Qwdx\right)\left(\frac{1}{|Q|}\int_Qw^{-\frac{1}{p-1}}dx\right)^{p-1},\quad (1<p<\infty),$$
and
$$[w]_{A_1}=\sup_{Q}\left(\frac{1}{|Q|}\int_Qwdx\right)(\inf_Qw)^{-1},$$
where the supremum is taken over all cubes $Q\subset {\mathbb R}^n$.
Sharp bounds for $L^p(w)$ operator norms in terms of $[w]_{A_p}$ have been recently found
for many central operators in Harmonic Analysis (see, e.g., \cite{B,CMP,Hyt1,L2,L3,P}).
A relatively simple approach to such bounds based on local mean oscillation estimates was developed in \cite{CMP,Hyt2,L1,L2,L3}. For the sake of comparison with the maximally modulated case treated in this article, we briefly review the relevant definitions and results. A Calder\'on-Zygmund operator on ${\mathbb R}^n$ is an $L^2$ bounded integral operator represented as
$$Tf(x)=\int_{{\mathbb R}^n}K(x,y)f(y)dy,\quad x\not\in\text{supp}\,f,$$
with kernel $K$ satisfying the following growth and smoothness conditions:
\begin{enumerate}
\renewcommand{\labelenumi}{(\roman{enumi})}
\item
$|K(x,y)|\le \frac{c}{|x-y|^n}$ for all $x\not=y$;
\item
$|K(x,y)-K(x',y)|+|K(y,x)-K(y,x')|\le
c\frac{|x-x'|^{\d}}{|x-y|^{n+\d}}$
for some $0<\d\le 1$ when $|x-x'|<|x-y|/2$.
\end{enumerate}
The sharp weighted bounds for standard Calder\'on-Zygmund operators as above can then be formulated as follows.
\begin{theoremA} Let $T$ be a Calder\'on-Zygmund operator on ${\mathbb R}^n$.
\begin{enumerate}
\renewcommand{\labelenumi}{(\roman{enumi})}
\item For any $1\le q<p<\infty$,
$$\|T\|_{L^p(w)}\le c(n,T,q,p)[w]_{A_q},$$
and in the case $q=1$, $c(n,T,1,p)=c(n,T)pp'$;
\item for any $1<p<\infty$,
$$\|T\|_{L^p(w)}\le c(n,T,p)[w]_{A_p}^{\max\big(1,\frac{1}{p-1}\big)}.$$
\end{enumerate}
\end{theoremA}

Part (i) for $q=1$ was obtained by Lerner-Ombrosi-Perez \cite{LOP1,LOP2}, and later Duoandikoetxea \cite{D} showed that the result for $q=1$ can be
self-improved by extrapolation to any $1<q<p$. The sharp dependence of $c(n,T,1,p)$ on $p$ is important for a weighted weak-$L^1$ bound of $T$ in terms
of $[w]_{A_1}$ \cite{LOP2}.
Part (ii) (known as the $A_2$ conjecture) is a more difficult result. First it was proved by
Petermichl \cite{P} for the Hilbert transform, and recently Hyt\"onen \cite{Hyt1} obtained (ii) for general Calder\'on-Zygmund operators.
A proof of Theorem A based on local mean oscillation estimates was found in \cite{L3,L4}. Observe that
for $p\ge 2$, (i) follows from (ii) but for $1<p<2$, (i) and (ii) are independent results.

In this article we apply the ``local mean oscillation estimate" approach of \cite{L3,L4} to
the Carleson operator ${\mathcal C}$ of \eqref{Carl} and its lacunary version  ${\mathcal C}_{lac}$, defined as
\begin{equation}{\mathcal C}_{lac}(f)(x)=\sup_{\xi\in \Xi
}|H({\mathcal M}^{\xi}f)(x)|,\label{Carllac}\end{equation}
where $\Xi\subset \mathbb R $ is any fixed $\theta$-lacunary set, in the sense that
$$
\inf_{\xi\neq \xi' \in \Xi} \frac{|\xi-\xi'|}{|\xi|} = 1-\textstyle \frac{1}{\theta} >0;
$$
for instance, one can take $\Xi=\{\pm\theta^k: k \in \mathbb Z\}$.
Our results involving ${\mathcal C},{\mathcal C}_{lac}$ will be derived as corollaries of the   more general Theorem \ref{mainr}, which is formulated in the framework of the maximally modulated singular integrals
studied by Grafakos-Martell-Soria \cite{GMS}; precise definitions follow. We remark that, unlike the results of \cite{DL,HY,GMS}, our focus is on the (possibly) sharp dependence of the $L^p(w)$ operator norms in terms of the $A_p$ constants of the weight.

Let ${\mathcal F}=\{\phi_{\a}\}_{\a\in A}$ be a family of real-valued measurable functions indexed by some set $A$, and let $T$ be a Calder\'on-Zygmund
operator. Then the maximally modulated Calder\'on-Zygmund operator $T^{\mathcal F}$ is defined by
$$T^{\mathcal F}f(x)=\sup_{\a\in A}|T({\mathcal M}^{\phi_{\a}}f)(x)|,$$
where ${\mathcal M}^{\phi_{\a}}f(x)={\rm e}^{2\pi i\phi_{\a}(x)}f(x)$.

The article  \cite{GMS} develops the weighted theory of such operators  under the \emph{a priori} assumption of a family of restricted weak-type $L^p$ bounds (see \eqref{restr} below) with controlled dependence of the constants when $p\to 1^+$, which, via a generalization of   the  approximation procedures from \cite{A,SS},      entails unrestricted boundedness $ T^{\mathcal F}: X \to L^{1,\infty}$ on appropriate (local) Orlicz spaces $X$ near $L^1$.  Such consequence of the \emph{a priori} assumption is, in fact, what is actually used in the derivation of weighted bounds for $T^{\mathcal F}$. The authors of \cite{GMS} do not explicitly state their estimates in terms of the $A_p$ characteristic of the weight; however,  an inspection of their proofs   shows that, loosely speaking, the  $L^p(w)$ constant  will not depend sharply on    $[w]_{A_p}$ unless the Orlicz space $X$ such that $ T^{\mathcal F}: X \to L^{1,\infty}$ is (essentially) the best possible.

When $T^{\mathcal F}=\mathcal C$ (resp.\ $T^{\mathcal F}=\mathcal C_{lac}$), this is the same as  the  largest  (in a suitable sense, see \cite{CMR}) Orlicz subspace $X \hookrightarrow L^1(\mathbb T) $ (resp.\ $X_{lac}$ from now on)  of pointwise convergence of  Fourier partial sums  (resp.\ of pointwise convergence along lacunary subsequences).
The best known result for $\mathcal C$, due to Antonov~\cite{A}, is that  $  L\log L\log\log\log L(\mathbb T) \hookrightarrow X$; this is also recovered in \cite{SS,GMS,LIE2}. In the lacunary case, the current best result (see \cite{LIE} and \cite{DP})  is that
 \begin{equation} \label{lacemb}  L\log\log L\log\log\log\log L(\mathbb T) \hookrightarrow X_{lac}.
\end{equation}
  The present results strongly suggest that the improvements   \begin{equation} \label{theconjemb}  L \log L (\mathbb T) \hookrightarrow X, \qquad  L\log\log L (\mathbb T) \hookrightarrow X_{lac} \end{equation} actually hold. The second embedding in \eqref{theconjemb}  (which would be sharp)  is explicitly conjectured by Konyagin in \cite{K2}, while the first is a widespread conjecture: see \cite{LIE2,DP2} for some recent articles on the subject.

Inspired by the recent approach to the study of the $p=1$ endpoint behavior of the operators $\mathcal C,\mathcal C_{lac}$ via (unrestricted) weak-type $L^p$ bounds of \cite{DP,DP2},   we choose to work under the (formally) stronger  \emph{a priori} assumption   of    the   family of weak-type inequalities
\begin{equation}\label{cond1}
\|T^{\mathcal F}(f)\|_{L^{p,\infty}}\lesssim \psi(p')\|f\|_{p}, \qquad (1<p\leq 2)
\end{equation}
where $\psi$ is a non-decreasing function on $[1,\infty)$. The key novelty compared to the previous approach relying on Orlicz bounds near $L^1$ is that sharp $[w]_{A_p}$ dependence is now related to sharp weak-$L^p$ bounds for $T^{\mathcal F}$, which are not affected by   losses introduced by the log-convexity of $L^{1,\infty}$ \cite{KAL}. These losses are, in fact,  the culprit for the triple and quadruple log in the current best results towards \eqref{theconjemb}.

Our  main result is the following: for convenience, we
denote by $S_0({\mathbb R}^n)$ the class of measurable functions on ${\mathbb R}^n$ such that
$$\mu_{f}(\la)=|\{x\in {\mathbb R}^n:|f(x)|>\la\}|<\infty\qquad \forall \lambda >0 .$$
\begin{theorem}\label{mainr} Let $T^{\mathcal F}$ be a maximally modulated Calder\'on-Zygmund operator
satisfying \emph{(\ref{cond1})}.
\begin{enumerate}
\renewcommand{\labelenumi}{(\roman{enumi})}
\item For any $1\le q<p<\infty$,
$$\|T^{\mathcal F}f\|_{L^p(w)}\le c(n,T^{\mathcal F},q,p)[w]_{A_q}\|f\|_{L^p(w)},$$
and in the case $q=1$, $c(n,T^{\mathcal F},1,p)=c(n,T^{\mathcal F})pp'\psi(3p')$.
\item
Then for any $1<p<\infty$,
$$\|T^{\mathcal F}f\|_{L^p(w)}\le c(n,T^{\mathcal F},p)\psi\Big(c(p,n)[w]_{A_p}^{\frac{1}{p-1}}\Big)[w]_{A_p}^{\max(1,\frac{1}{p-1})}.$$
\end{enumerate}
Both estimates in (i) and (ii) are understood in the sense that they hold for any $f\in L^p(w)$ for which
$T^{\mathcal F}f\in S_0$.
\end{theorem}

We observe that the last sentence in Theorem \ref{mainr} can be removed if it is additionally known that $T^{\mathcal F}f\in S_0$ for some dense subset in $L^p(w)$, for instance, for Schwartz functions. In particular, this obviously holds if $T^{\mathcal F}$ is of weak type $(r_0,r_0)$ for some $r_0>1$. Hence, there is no need for the last sentence in  Theorem \ref{mainr} for  the Carleson operator and its lacunary version.

Observe also that if $T^{\mathcal F}=T$ is the standard Calder\'on-Zygmund operator, then $\psi=1$, and hence Theorem \ref{mainr}
contains Theorem A as a particular case.

Let us now turn to how assumption \eqref{cond1} is satisfied in the concrete cases we are interested in. We will prove in Section \ref{sectorlicz} that
\begin{equation}\label{wbcarl}\| \mathcal C(f)\|_{ {p,\infty}}\lesssim p'\log\log({\rm e}^{\rm e}+ p')\|f\|_{p}, \qquad (1<p\leq 2),
\end{equation}
as a consequence of Antonov's result:  this seems to be the best available dependence in the literature.  Our methods, based on the local sharp maximal function, are general enough to derive an appropriate form of \eqref{cond1} assuming some local Orlicz space boundedness into $L^{1,\infty}$.
In the  lacunary case, the bound \begin{equation}\label{wbcarllac}\| \mathcal C_{lac}(f)\|_{ {p,\infty}}\lesssim\log ( {\rm e} +p')\|f\|_{p}, \qquad (1<p\leq 2),
\end{equation}
the implicit constant depending only on the lacunarity constant $\theta$ of the associated sequence, can be obtained by suitably modifying the proof of the main result in \cite{LIE}: we send to the preprint \cite{DP2} for a thorough account of the necessary changes. We also note that \eqref{wbcarllac} entails \eqref{lacemb} as an easy consequence (see \cite{DP} for details).
In view of \eqref{wbcarl}-\eqref{wbcarllac}, taking $\psi(t)=t\log\log ({\rm e}^{\rm e} +t)$ ($\psi(t)=\log( {\rm e}+ t)$ respectively) in Theorem \ref{mainr} yields immediately the following corollaries.
\begin{cor}\label{carl} Let $\mathcal C$ be the Carleson operator.
\begin{enumerate}
\renewcommand{\labelenumi}{(\roman{enumi})}
\item For any $1\le q<p<\infty$,
$$\|\mathcal C\|_{L^p(w)}\le c(q,p)[w]_{A_q},$$
and in the case $q=1$, $c(1,p)\simeq p (p')^2\log\log({\rm e}^{\rm e}p')$;
\item for any $1<p<\infty$,
$$\|\mathcal C\|_{L^p(w)}\le c(p)[w]_{A_p}^{\max\big(p',\frac{2}{p-1}\big)} \log\log ({\rm e}^{\rm e}+[w]_{A_p}).$$
\end{enumerate}
\end{cor}

\begin{cor}\label{carllac} Let ${\mathcal C}_{lac}$ be the lacunary Carleson operator.
\begin{enumerate}
\renewcommand{\labelenumi}{(\roman{enumi})}
\item For any $1\le q<p<\infty$,
$$\|{\mathcal C}_{lac}\|_{L^p(w)}\le c(q,p)[w]_{A_q},$$
and in the case $q=1$, $c(1,p)\simeq pp'\log(  {\rm e}+ p')$;
\item for any $1<p<\infty$,
$$\|{\mathcal C}_{lac}\|_{L^p(w)}\le c(p)[w]_{A_p}^{\max\big(1,\frac{1}{p-1}\big)} \log({\rm e}+ [w]_{A_p}).$$
\end{enumerate}
\end{cor}

Since the linear $[w]_{A_q}, 1\le q<p,$ bound is sharp for the Hilbert transform, it is obviously sharp also for ${\mathcal C}$ and ${\mathcal C}_{lac}$.
We thouroughly discuss sharpness in terms of $[w]_{A_p}$ of the points (ii) in Section \ref{sect6} below; here, we mention that being able to drop the $\log\log$ term in \eqref{wbcarl} would produce the same effect in (ii) of Corollary \ref{carl}. For the Walsh analogue $\W$ of the Carleson operator, we are able to do so: relying on the analogue of condition \eqref{cond1}
\begin{equation}
\label{cond1walsh}
\|\W f\|_{L^{p,\infty}(\mathbb T)} \lesssim p'\|f\|_{L^{p}(\mathbb T)},
\end{equation}
 which has been established in \cite{DP2},
 we prove the following weighted theorem.
  \begin{theorem}\label{Walshthm} Let $\mathcal W$ be the Walsh-Carleson maximal operator, defined in \eqref{walshcarl} below.
\begin{enumerate}
\renewcommand{\labelenumi}{(\roman{enumi})}
\item For any $1\le q<p<\infty$,
$$\|\mathcal W\|_{L^p(w)}\le c(q,p)[w]_{A_q},$$
and in the case $q=1$, $c(1,p)\simeq p (p')^2$;
\item for any $1<p<\infty$,
$$\|\mathcal W\|_{L^p(w)}\le c(p)[w]_{A_p}^{\max\big(p',\frac{2}{p-1}\big)}.$$
\end{enumerate}
\end{theorem}

The article is organized as follows. In Section 2, we obtain a local mean oscillation estimate of $T^{\mathcal F}$, and the corresponding bound by
dyadic sparse operators. Using this result, we prove Theorem \ref{mainr} in Sections \ref{sect3} and \ref{sect4}.  Section \ref{secwalsh} contains the proof of Theorem \ref{Walshthm}; in Section \ref{sectorlicz}, we relate assumption (\ref{cond1}) to local Orlicz bounds.

Throughout the paper, we use the notation $A\lesssim B$ to indicate that there is a constant $c$, independent of the important parameters, such that $A\leq cB$.
We write $A\simeq B$ when $A\lesssim B$ and $B\lesssim A$.

\vskip 2mm
{\bf Acknowledgement.}  The second author is very grateful to Loukas Grafakos for his useful comments
on the Carleson operator.

\section{An estimate of $T^{\mathcal F}$ by dyadic sparse operators}
\subsection{A local mean oscillation estimate}
By a general dyadic grid ${\mathscr{D}}$ we mean a collection of
cubes with the following properties: (i)
for any $Q\in {\mathscr{D}}$ its sidelength $\ell_Q$ is of the form
$2^k, k\in {\mathbb Z}$; (ii) $Q\cap R\in\{Q,R,\emptyset\}$ for any $Q,R\in {\mathscr{D}}$;
(iii) the cubes of a fixed sidelength $2^k$ form a partition of ${\mathbb
R}^n$.

Denote the standard dyadic grid $\{2^{-k}([0,1)^n+j), k\in{\mathbb Z}, j\in{\mathbb Z}^n\}$
by ${\mathcal D}$.
Given a cube $Q_0$, denote by ${\mathcal D}(Q_0)$ the set of all
dyadic cubes with respect to $Q_0$, that is, the cubes from ${\mathcal D}(Q_0)$ are formed
by repeated subdivision of $Q_0$ and each of its descendants into $2^n$ congruent subcubes.

We say that a family of cubes ${\mathcal S}$ is sparse if for any cube $Q\in {\mathcal S}$ there is a
measurable subset $E(Q)\subset Q$ such that $|Q|\le 2|E(Q)|$, and the sets $\{E(Q)\}_{Q\in {\mathcal S}}$
are pairwise disjoint.

Given a measurable function $f$ on ${\mathbb R}^n$ and a cube $Q$,
the local mean oscillation of $f$ on $Q$ is defined by
$$\o_{\la}(f;Q)=\inf_{c\in {\mathbb R}}
\big((f-c)\chi_{Q}\big)^*\big(\la|Q|\big)\quad(0<\la<1),$$
where $f^*$ denotes the non-increasing rearrangement of $f$.

By a median value of $f$ over $Q$ we mean a possibly nonunique, real
number $m_f(Q)$ such that
$$\max\big(|\{x\in Q: f(x)>m_f(Q)\}|,|\{x\in Q: f(x)<m_f(Q)\}|\big)\le |Q|/2.$$

The following result was proved in \cite{L1}; in its current refined version given below it can be found in \cite{Hyt2}.

\begin{theorem}\label{lmoes} Let $f$ be a measurable function on
${\mathbb R}^n$ and let $Q_0$ be a fixed cube. Then there exists a
(possibly empty) sparse family ${\mathcal S}$ of cubes from ${\mathcal D}(Q_0)$ such that for a.e. $x\in Q_0$,
$$
|f(x)-m_f(Q_0)|\le 2\sum_{Q\in {\mathcal S}}
\o_{\frac{1}{2^{n+2}}}(f;Q)\chi_{Q}(x).
$$
\end{theorem}
Given a measurable function $f$ on ${\mathbb R}^n$, define the local sharp maximal function $M_{\la}^{\#}f$ by
$$M_{\la}^{\#}f(x)=\sup_{Q\ni x}\o_{\la}(f;Q),$$
where the supremum is taken over all cubes $Q\subset {\mathbb R}^n$ containing the point $x$.

\begin{lemma}\label{locsh} For any $f\in S_0$ and for all $p>0$,
\begin{equation}\label{locshest}
\|f\|_{L^{p,\infty}}\le 3p\|M_{\la}^{\#}f\|_{L^{p,\infty}},
\end{equation}
where $\la$ depends only on $n$.
\end{lemma}

\begin{proof} We use the following rearrangement estimate proved in \cite{L}:
\begin{equation}\label{rear}
f^*(t)\le 2(M_{\la_n}^{\#}f)^*(2t)+f^*(2t)\quad(t>0).
\end{equation}
Since $f\in S_0$ is equivalent to that $f^*(\infty)=0$, iterating (\ref{rear}) yields
\begin{eqnarray*}
f^*(t)&\le& 2\sum_{k=0}^{\infty}(M_{\la_n}^{\#}f)^*(2^kt)\le \frac{2}{\log 2}
\sum_{k=0}^{\infty}\int_{2^{k-1}t}^{2^kt}(M_{\la_n}^{\#}f)^*(s)\frac{ds}{s}\\
&\le& 3\int_t^{\infty}(M_{\la_n}^{\#}f)^*(s)\frac{ds}{s}\le 3pt^{-1/p}\|M_{\la_n}^{\#}f\|_{L^{p,\infty}},
\end{eqnarray*}
which proves (\ref{locshest}).
\end{proof}

Observe that Lemma \ref{locsh} was obtained in \cite{JT} by a different method with
an exponential dependence on $p$.

\subsection{An application to $T^{\mathcal F}$} We now apply Theorem \ref{lmoes} to $T^{\mathcal F}$.
Given a cube $Q$, we denote $\bar Q=2\sqrt n Q$.

\begin{lemma}\label{oscest}
Suppose $T^{\mathcal F}$ satisfies \emph{(\ref{cond1})}. Then for any cube $Q\subset {\mathbb R}^n$ and for all $1<r\le 2$,
\begin{eqnarray}
\o_{\la}(T^{\mathcal F}f;Q)&\lesssim& \psi(r')\left(\frac{1}{|\bar Q|}\int_{\bar Q}|f|^r\right)^{1/r}\label{osc}\\
&+&\sum_{m=0}^{\infty}\frac{1}{2^{m\d}}\left(\frac{1}{|2^mQ|}\int_{2^mQ}|f|\right).\nonumber
\end{eqnarray}
\end{lemma}

\begin{proof}
This result is a minor modification of \cite[Prop. 2.3]{L3}, and it is essentially contained in \cite[Prop. 4.1]{GMS}.
We  briefly outline the main steps of proof.

Set $f_1=f\chi_{\bar Q}$ and $f_2=f-f_1$. Let $x\in Q$ and let $x_0$ be the center of $Q$. Then
\begin{eqnarray*}
&&|T^{\mathcal F}(f)(x)-T^{\mathcal F}(f_2)(x_0)|\\
&&=\Big|\sup_{\a\in A}
|T({\mathcal M}^{\phi_{\a}}f)(x)|-\sup_{\a\in A}|T({\mathcal M}^{\phi_{\a}}f_2)(x_0)|\Big|\\
&&\le \sup_{\a\in A}|T({\mathcal M}^{\phi_{\a}}f)(x)-T({\mathcal M}^{\phi_{\a}}f_2)(x_0)|\\
&&\le T^{\mathcal F}(f_1)(x)+\sup_{\a\in A}\|T({\mathcal M}^{\phi_{\a}}f_2)(\cdot)-T({\mathcal M}^{\phi_{\a}}f_2)(x_0)\|_{L^{\infty}(Q)}.
\end{eqnarray*}

Exactly as in \cite[Prop. 2.3]{L3}, by the kernel assumption,
\begin{eqnarray*}
&&\sup_{\a\in A}\|T({\mathcal M}^{\phi_{\a}}f_2)(\cdot)-T({\mathcal M}^{\phi_{\a}}f_2)(x_0)\|_{L^{\infty}(Q)}\\
&&\le \int_{{\mathbb R}^n\setminus \bar Q}|f(y)|\|K(\cdot,y)-K(x_0,y)\|_{L^{\infty}(Q)}dy\\
&&\lesssim \sum_{m=0}^{\infty}\frac{1}{2^{m\d}}\left(\frac{1}{|2^mQ|}\int_{2^mQ}|f|\right).
\end{eqnarray*}
For the local part, by (\ref{cond1}),
$$\big(T^{\mathcal F}(f_1)\chi_Q\big)^*(\la|Q|)\lesssim \psi(r')\left(\frac{1}{|\bar Q|}\int_{\bar Q}|f|^r\right)^{1/r}.$$
Combining this estimate with the two previous ones, and taking $c=T^{\mathcal F}(f_2)(x_0)$ in the definition of $\o_{\la}(T^{\mathcal F}f;Q)$
proves (\ref{osc}).
\end{proof}

Given a sparse family ${\mathcal S}$, define the operators ${\mathcal A}_{r, \mathcal S}$ and ${\mathcal T}_{\mathcal S,m}$
respectively by
$${\mathcal A}_{r, \mathcal S}f(x)=\sum_{Q\in {\mathcal S}}\left(\frac{1}{|\bar Q|}\int_{\bar Q}|f|^r\right)^{1/r}\chi_Q(x).$$
and
$$
{\mathcal T}_{\mathcal S,m}f(x)=\sum_{Q\in {\mathcal S}}\left(\frac{1}{|2^mQ|}\int_{2^mQ}|f|\right)\chi_Q(x)
$$

\begin{lemma}\label{kest} Suppose $T^{\mathcal F}$ satisfies \emph{(\ref{cond1})}. Let $1<p<\infty$ and let $w$ be an arbitrary weight. Then
\begin{equation}\label{rightA}
\|T^{\mathcal F}f\|_{L^p(w)}\lesssim  \inf_{1<r\le 2}\Big\{\psi(r')\sup_{{\mathscr{D}},{\mathcal S}}\|{\mathcal A}_{r, \mathcal S}f\|_{L^p(w)}\Big\}
\end{equation}
for any $f$ for which $T^{\mathcal F}f\in S_0$, where the supremum is taken over all dyadic grids ${\mathscr{D}}$ and all
sparse families ${\mathcal S}\subset {\mathscr{D}}$.
\end{lemma}

\begin{proof} Let $Q_0\in {\mathcal D}$. Combining Theorem \ref{lmoes} with Lemma \ref{oscest}, we obtain that there exists a sparse family ${\mathcal S}\subset {\mathcal D}$ such that for a.e. $x\in Q_0$,
\begin{equation}\label{intes}
|T^{\mathcal F}f(x)-m_{T^{\mathcal F}f}(Q_0)|\lesssim \psi(r'){\mathcal A}_{r,\mathcal S}f(x)+\sum_{m=0}^{\infty}\frac{1}{2^{m\d}}{\mathcal T}_{\mathcal S,m}f(x).
\end{equation}

If $T^{\mathcal F}f\in S_0$, then $m_{T^{\mathcal F}f}(Q)\to 0$ as $|Q|\to \infty$. Hence, letting $Q_0$ to anyone of $2^n$ quadrants and
using (\ref{intes}) along with Fatou's lemma, we get
$$\|T^{\mathcal F}f\|_{L^p(w)}\lesssim \psi(r')\sup_{{\mathcal S}\subset {\mathcal D}}\|{\mathcal A}_{r, \mathcal S}f\|_{L^p(w)}+
\sum_{m=0}^{\infty}\frac{1}{2^{m\d}}\sup_{{\mathcal S}\subset {\mathcal D}}\|{\mathcal T}_{\mathcal S,m}f\|_{L^p(w)}.$$
It was shown in \cite{L3} that
$$\sup_{{\mathcal S}\subset {\mathcal D}}\|{\mathcal T}_{\mathcal S,m}f\|_{L^p(w)}\lesssim
m\sup_{{\mathscr{D}},{\mathcal S}}\|{\mathcal T}_{\mathcal S,0}f\|_{L^p(w)}.
$$
Next, by H\"older's inequality,
$$\|{\mathcal T}_{\mathcal S,0}f\|_{L^p(w)}\lesssim \|{\mathcal A}_{r,\mathcal S}f\|_{L^p(w)}.$$
Combining this with the two previous estimates completes the proof.
\end{proof}

\begin{remark}\label{rem}
Observe that the implicit constant in (\ref{rightA}) depends only on $T^{\mathcal F}$ and $n$. In fact,
one can replace $L^p(w)$ in this inequality by an arbitrary Banach function space $X$, exactly the same as for standard Calder\'on-Zygmund operators (see \cite{L3}).
\end{remark}

\section{Proof of Theorem \ref{mainr}, part (i)}\label{sect3}
For $s>0$ let $M_sf(x)=M(|f|^s)(x)^{1/s}$, where $M$ is the Hardy-Littlewood maximal operator.
We will use several results from \cite{LOP1} which can be summarized as follows
(note that part (ii) of Proposition \ref{sum} is contained in the proof of \cite[Lemma 3.3]{LOP1}).

\begin{prop}\label{sum} The following estimates hold:
\begin{enumerate}
\renewcommand{\labelenumi}{(\roman{enumi})}
\item if $w\in A_1$ and $s_w=1+\frac{1}{2^{n+1}[w]_{A_1}}$, then
$$M_{s_{w}}w(x)\le 2[w]_{A_1}w(x);$$
\item for any $p>1$ and $1<s<2$,
$$\|Mf\|_{L^{p'}((M_sw)^{-\frac{1}{p-1}})}\le c(n)p\Big(\frac{1}{s-1}\Big)^{1-1/ps}\|f\|_{L^{p'}(w^{-\frac{1}{p-1}})}.$$
\end{enumerate}
\end{prop}

We also recall the Fefferman-Stein inequality \cite{FS}:
\begin{equation}\label{fs}
\|Mf\|_{L^p(w)}\le c(n)p'\|f\|_{L^p(Mw)}\quad (1<p<\infty),
\end{equation}
and the Coifman inequality \cite{CR}:
\begin{equation}\label{coif}
M\Big((Mf)^{\a}\Big)(x)\le c(n)\frac{1}{1-\a}Mf(x)^{\a}\quad(0<\a<1).
\end{equation}

\begin{proof}[Proof of Theorem \ref{mainr}, part (i)]
By extrapolation (\cite[Cor. 4.3.]{D}), it suffices to consider only the case $q=1$. Hence, our aim is to show
that for any $1<p<\infty$,
$$\|T^{\mathcal F}f\|_{L^p(w)}\le c(n,T^{\mathcal F})pp'\psi(3p')[w]_{A_1}\|f\|_{L^p(w)}.$$
By Lemma \ref{kest} (see also Remark \ref{rem}), this will be a consequence of
\begin{equation}\label{a1norm}
\inf_{1<r\le 2}\Big\{\psi(r')\sup_{{\mathscr{D}},{\mathcal S}}\|{\mathcal A}_{r, \mathcal S}f\|_{L^p(w)}\Big\}
\le c(n)pp'\psi(3p')[w]_{A_1}\|f\|_{L^p(w)}.
\end{equation}

In order to prove (\ref{a1norm}), we obtain the following estimate: for any $1<p<\infty$ and $1<r<\min(2,\frac{p+1}{2})$,
\begin{equation}\label{estar}
\sup_{{\mathscr{D}},{\mathcal S}}\|{\mathcal A}_{r, \mathcal S}f\|_{L^p(w)}\le c(n)\Big(\Big(\frac{p+1}{2r}\Big)'\Big)^{1/r}p[w]_{A_1}\|f\|_{L^p(w)}.
\end{equation}
Assuming for a moment that  (\ref{estar}) holds true, (\ref{a1norm}) follows easily, since
\begin{equation}\label{inf1r}
\inf_{1<r\le 2}\psi(r')\Big(\Big(\frac{p+1}{2r}\Big)'\Big)^{1/r}\le cp'\psi(3p')
\end{equation}
for some absolute $c>0$. To get (\ref{inf1r}), observe that in the case $p\ge 3$ one can take
$r=3/2$, namely,
$$
\inf_{1<r\le 2}\psi(r')\Big(\Big(\frac{p+1}{2r}\Big)'\Big)^{1/r}\le \psi(3)\Big(\Big(\frac{p+1}{3}\Big)'\Big)^{2/3}\le c\psi(3p').
$$
If $p<3$, we take $r=\frac{p+3}{4}$, and then
\begin{eqnarray*}
\inf_{1<r\le 2}\psi(r')\Big(\Big(\frac{p+1}{2r}\Big)'\Big)^{1/r}&\le& \psi\Big(\frac{p+3}{4(p-1)}\Big)\Big(\Big(\frac{2(p+1)}{p+3}\Big)'\Big)^{\frac{4}{p+3}}\\
&\le& c\frac{1}{p-1}\psi(2p').
\end{eqnarray*}
Combining both cases yields (\ref{inf1r}).

We turn to the proof of (\ref{estar}).
Fix a dyadic grid ${\mathscr{D}}$ and a sparse family ${\mathcal S}\subset {\mathscr{D}}$. One can assume that $f\ge 0$.
We linearize the operator ${\mathcal A}_{r, \mathcal S}$ as follows. For any $Q\in {\mathcal S}$ there exists $g_{(Q)}$ supported in $\bar Q$ such that
$\frac{1}{|\bar Q|}\int_{\bar Q}g_{(Q)}^{r'}=1$ and
$$
\left(\frac{1}{|\bar Q|}\int_{\bar Q}f^r\right)^{1/r}=\frac{1}{|\bar Q|}\int_{\bar Q}fg_{(Q)}.
$$
Define now the linear operator $L$ by
$$L(h)(x)=\sum_{Q\in {\mathcal S}}\left(\frac{1}{|\bar Q|}\int_{\bar Q}hg_{(Q)}\right)\chi_Q(x).$$
Then $L(f)={\mathcal A}_{r, \mathcal S}(f)$, and hence, in order to prove (\ref{estar}), it suffices to show that
\begin{equation}\label{suff}
\|L(h)\|_{L^p(w)}\le c(n)\Big(\Big(\frac{p+1}{2r}\Big)'\Big)^{1/r}p[w]_{A_1}\|h\|_{L^p(w)},
\end{equation}
uniformly in $g_{(Q)}$.

Exactly as it was done in \cite{LOP1}, we have that (\ref{suff}) will follow from
\begin{equation}\label{suff1}
\|L(h)\|_{L^p(w)}\le c(n)\Big(\Big(\frac{p+1}{2r}\Big)'\Big)^{1/r}p\Big(\frac{1}{s-1}\Big)^{1-1/ps}\|h\|_{L^p(M_sw)},
\end{equation}
where $1<s<2$. Indeed, taking here $s=s_w=1+\frac{1}{2^{n+1}[w]_{A_1}}$, by (i) of Proposition \ref{sum},
$$
\Big(\frac{1}{s_w-1}\Big)^{1-1/ps_w}\|h\|_{L^p(M_{s_w}w)}\le c(n)[w]_{A_1}\|h\|_{L^p(w)},
$$
which yields (\ref{suff}).

Let $L^*$ denote the formal adjoint of $L$. By duality, (\ref{suff1}) is equivalent to
$$
\|L^*(h)\|_{L^{p'}((M_sw)^{-\frac{1}{p-1}})}\le
c(n)\Big(\big(\textstyle\frac{p+1}{2r}\big)'\Big)^{1/r}p\Big(\textstyle\frac{1}{s-1}\Big)^{1-1/ps}\|h\|_{L^{p'}(w^{-\frac{1}{p-1}})},
$$
which, by (ii) of Proposition \ref{sum}, is an immediate corollary of
\begin{equation}\label{suff2}
\|L^*(h)\|_{L^{p'}((M_sw)^{-\frac{1}{p-1}})}\le c(n)\Big(\Big(\frac{p+1}{2r}\Big)'\Big)^{1/r}\|Mh\|_{L^{p'}((M_sw)^{-\frac{1}{p-1}})}.
\end{equation}

We now prove (\ref{suff2}).
By duality, pick $\eta\ge 0$ such that $\|\eta\|_{L^p(M_sw)}=1$ and
$$\|L^*(h)\|_{L^{p'}((M_sw)^{-\frac{1}{p-1}})}=\int_{{\mathbb R}^n}L^*(h)\eta dx=\int_{{\mathbb R}^n}hL(\eta) dx.$$
Using H\"older's inequality and the sparseness of ${\mathcal S}$, we obtain
\begin{eqnarray*}
&&\int_{{\mathbb R}^n}hL(\eta) dx=\sum_{Q\in {\mathcal S}}\left(\frac{1}{|\bar Q|}\int_{\bar Q}\eta g_{(Q)}\right)\int_Qh\le
\sum_{Q\in {\mathcal S}}\left( \textstyle\frac{1}{|\bar Q|}\int_{\bar Q}\eta^r\right)^{1/r}\int_Qh\\
&&\le(2\sqrt n)^n\sum_{Q\in {\mathcal S}}\left(\frac{1}{|\bar Q|}\int_{\bar Q}\eta^r\right)^{1/r}\left(\frac{1}{|\bar Q|}\int_{\bar Q}h\right)|Q|\\
&&\le 2(2\sqrt n)^n \sum_{Q\in {\mathcal S}}
\left(\frac{1}{|\bar Q|}\int_{\bar Q}\Big((Mh)^{\frac{1}{p+1}}\eta\Big)^r\right)^{1/r}
\left(\frac{1}{|\bar Q|}\int_{\bar Q}h\right)^{\frac{p}{p+1}}|E(Q)|\\
&&\le 2(2\sqrt n)^n\sum_{Q\in {\mathcal S}}\int_{E(Q)}M_{r}((Mh)^{\frac{1}{p+1}}\eta)(Mh)^{\frac{p}{p+1}}dx\\
&&\le 2(2\sqrt n)^n \int_{{\mathbb R}^n}M_{r}((Mh)^{\frac{1}{p+1}}\eta)(Mh)^{\frac{p}{p+1}}dx.
\end{eqnarray*}
Next, by H\"older's inequality with the exponents $\rho=\frac{p+1}{2}$ and $\rho'=\frac{p+1}{p-1}$,
\begin{eqnarray*}
&&\int_{{\mathbb R}^n}M_{r}((Mh)^{\frac{1}{p+1}}\eta)(Mh)^{\frac{p}{p+1}}dx\\
&&=\int_{{\mathbb R}^n}M_{r}((Mh)^{\frac{1}{p+1}}\eta)(M_sw)^{\frac{1}{p+1}}(Mh)^{\frac{p}{p+1}}(M_sw)^{-\frac{1}{p+1}}dx\\
&&\le \|M_{r}((Mh)^{\frac{1}{p+1}}\eta)\|_{L^{\frac{p+1}{2}}((M_sw)^{1/2})}\|Mh\|_{L^{p'}((M_sw)^{-\frac{1}{p-1}})}^{\frac{p}{p+1}}.
\end{eqnarray*}
Combining (\ref{fs}) and (\ref{coif}) yields
\begin{eqnarray*}
&&\|M_{r}((Mh)^{\frac{1}{p+1}}\eta)\|_{L^{\frac{p+1}{2}}((M_sw)^{1/2})}\\
&&\le c(n)\Big(\Big(\frac{p+1}{2r}\Big)'\Big)^{1/r}
\|(Mh)^{\frac{1}{p+1}}\eta\|_{L^{\frac{p+1}{2}}(M(M_sw)^{1/2})}\\
&&\le c(n)\Big(\Big(\frac{p+1}{2r}\Big)'\Big)^{1/r}\|(Mh)^{\frac{1}{p+1}}\eta\|_{L^{\frac{p+1}{2}}((M_sw)^{1/2})}.
\end{eqnarray*}
Using again H\"older's inequality with $\rho=2p'$ and $\rho'=\frac{2p}{p+1}$, we get
\begin{eqnarray*}
&&\|(Mh)^{\frac{1}{p+1}}\eta\|_{L^{\frac{p+1}{2}}((M_sw)^{1/2})}\\
&&=\left(\int_{{\mathbb R}^n}\Big((Mh)^{\frac{1}{2}}(M_sw)^{-\frac{1}{2p}}\Big)\Big(\eta^{\frac{p+1}{2}}(M_sw)^{\frac{p+1}{2p}}\Big)dx\right)^{\frac{2}{p+1}}\\
&&\le \|Mh\|_{L^{p'}((M_sw)^{-\frac{1}{p-1}})}^{\frac{1}{p+1}}\|\eta\|_{L^p(M_sw)}=\|Mh\|_{L^{p'}((M_sw)^{-\frac{1}{p-1}})}^{\frac{1}{p+1}}.
\end{eqnarray*}
Combining this estimate with the three previous ones yields (\ref{suff2}), and therefore the theorem is proved.
\end{proof}

\begin{remark}\label{rem5}
The key ingredient in the proof of the linear $[w]_{A_1}$ bound for usual Calder\'on-Zygmund operators $T$ in \cite{LOP1,LOP2}
is a Coifman-type estimate relating the adjoint operator $T^*$ and the Hardy-Littlewood maximal operator $M$. This method relies crucially on the fact that
that $T^*$ is essentially the same operator as $T$. However, this is not the case with the Carleson operator ${\mathcal C}$. Indeed,
taking an arbitrary measurable function $\xi(\cdot)$, we can consider the standard linearization of ${\mathcal C}$ given by
$${\mathcal C}_{\xi(\cdot)}(f)(x)=H({\mathcal M}^{\xi(x)}f)(x).$$
It is difficult to expect that its adjoint ${\mathcal C}_{\xi(\cdot)}^*$ can be related (uniformly in $\xi(\cdot)$) with $M$ (or even with a bigger maximal operator)
either via good-$\lambda$ or by a sharp function estimate. Indeed, such a relation would imply that $\|{\mathcal C}_{\xi(\cdot)}^*\|_{L^p}\lesssim p$
as $p\to\infty$ (since $\|f\|_{L^p}\lesssim p\|f^{\#}\|_{L^p}$ as $p~\to~\infty$, where $f^{\#}$ is the sharp function), which in turn means that $\|{\mathcal C}\|_{L^p} \lesssim \frac{1}{p-1}$ as $p\to 1$. Such a result, significantly stronger than   the best known dependence of $\|{\mathcal C}\|_{L^p}$ as $p\to 1$ (which can be read from eq.\ \eqref{strongboundlie} below),  would entail in particular that ${\mathcal C}:L\log L(\mathbb T) \to L^1(\mathbb T)$ just like the Hilbert transform. This bound is   a much stricter form of the conjectured \eqref{theconjemb},  and is not even known to hold  for the tamer  lacunary Carleson operator $\mathcal C_{lac}$. See Section \ref{sect6} for further discussion.

The operator $L^*$ defined in the proof can be viewed as a dyadic positive model of ${\mathcal C}_{\xi(\cdot)}^*$, and inequality
(\ref{suff2}) is a Coifman type estimate relating $L^*$ and $M$. By the reasons described above we cannot apply
the approach used in \cite{LOP1,LOP2} directly to (\ref{suff2}).
\end{remark}

\section{Proof of Theorem \ref{mainr}, part (ii)} \label{sect4}
We start with some preliminaries. Given a sparse family ${\mathcal S}$, define the operator ${\mathcal T}_{\mathcal S}$ by
$$
{\mathcal T}_{\mathcal S}f(x)=\sum_{Q\in {\mathcal S}}\left(\frac{1}{|\bar Q|}\int_{\bar Q}|f|\right)\chi_Q(x).
$$
This operator satisfies
\begin{equation}\label{cmp}
\|{\mathcal T}_{\mathcal S}\|_{L^p(w)}\lesssim [w]_{A_p}^{\max(1,\frac{1}{p-1})} \quad(1<p<\infty).
\end{equation}
For the fully dyadic version ${\mathcal T}_{\mathcal S,0}$ (introduced in Section 2) this estimate was proved in \cite{CMP}. The same proof with
minor modifications works for ${\mathcal T}_{\mathcal S}$ as well. Alternatively, one can use that
$$\|{\mathcal T}_{\mathcal S}\|_{L^p(w)}\lesssim \sup_{{\mathscr{D}},{\mathcal S}}\|{\mathcal T}_{\mathcal S,0}f\|_{L^p(w)},$$
proved in \cite{L3}, and subsequently apply the result for ${\mathcal T}_{\mathcal S,0}$.

Furthermore,  we recall that (see \cite{B})
\begin{equation}\label{buck}
\|M\|_{L^p(w)}\lesssim [w]_{A_p}^{\frac{1}{p-1}}\quad(1<p<\infty).
\end{equation}
Also, it is mentioned in \cite{B} (a detailed proof can be found in \cite{HPR}) that
\begin{equation}\label{bhpr}
\e=c(n,p)[w]_{A_p}^{1-p'}\Rightarrow\,\,[w]_{A_{p-\e}}\lesssim [w]_{A_p}.
\end{equation}

We will use the following proposition.
\begin{prop}\label{pr} Let $r>1$. Then for any cube $Q$,
$$\Big(\frac{1}{|Q|}\int_Q|f|^rdx\Big)^{1/r}\le \frac{1}{|Q|}\int_Q|f|dx+2^n(r-1)\frac{1}{|Q|}\int_QM_r(f\chi_Q)dx.$$
\end{prop}

\begin{proof}
By homogeneity, one can assume that $\frac{1}{|Q|}\int_Q|f|^r=1$. We use the classical estimate (see \cite{St}) saying that for $|f|_Q\le \a$,
$$\frac{1}{\a}\int_{\{x\in Q:|f(x)|>\a\}}|f(x)|dx\le 2^n|\{x\in Q:M(f\chi_Q)>\a\}|.$$
Replacing here $|f|$ by $|f|^r$ and $\a$ by $\a^r$, we obtain that for $\a\ge 1$,
$$\frac{1}{\a^r}\int_{\{x\in Q:|f(x)|>\a\}}|f(x)|^rdx\le 2^n|\{x\in Q:M_r(f\chi_Q)>\a\}|.$$
From this,
\begin{eqnarray*}
\frac{1}{r-1}\int_{\{x\in Q:|f|>1\}}(|f|^r-|f|)dx&=&\int_1^{\infty}\frac{1}{\a^r}\int_{\{x\in Q:|f(x)|>\a\}}|f(x)|^rdxd\a\\
&\le& 2^n\int_QM_r(f\chi_Q)dx.
\end{eqnarray*}
Combining this estimate with
$$\int_Q(|f|^r-|f|)dx\le \int_{\{x\in Q:|f|>1\}}(|f|^r-|f|)dx,$$
we obtain
$$\int_Q(|f|^r-|f|)dx\le 2^n(r-1)\int_QM_r(f\chi_Q)dx.$$
Hence,
$$1\le \frac{1}{|Q|}\int_Q|f|dx+2^n(r-1)\frac{1}{|Q|}\int_QM_r(f\chi_Q)dx,$$
which completes the proof.
\end{proof}

\begin{proof}[Proof of Theorem \ref{mainr}, part (ii)]
By Proposition \ref{pr},
$$
{\mathcal A}_{r, \mathcal S}f(x)\le {\mathcal T}_{\mathcal S}f(x)+2^n(r-1){\mathcal T}_{\mathcal S}(M_rf)(x).
$$
Combining this estimate with (\ref{cmp}) and (\ref{buck}) yields
\begin{eqnarray*}
\|{\mathcal A}_{r, \mathcal S}\|_{L^p(w)}&\le& \|{\mathcal T}_{\mathcal S}\|_{L^p(w)}+2^n(r-1)\|{\mathcal T}_{\mathcal S}\|_{L^p(w)}
\|M_r\|_{L^p(w)}\\
&\lesssim& [w]_{A_p}^{\max(1,\frac{1}{p-1})}+2^n(r-1)[w]_{A_p}^{\max(1,\frac{1}{p-1})}[w]_{A_{p/r}}^{\frac{r}{p-r}}.
\end{eqnarray*}
Take $r=\frac{p}{p-\e}$, where $\e$ is given by (\ref{bhpr}). Then we obtain
\begin{eqnarray*}
&&\inf_{1<r\le 2}\psi(r')\Big([w]_{A_p}^{\max(1,\frac{1}{p-1})}+2^n(r-1)[w]_{A_p}^{\max(1,\frac{1}{p-1})}[w]_{A_{p/r}}^{\frac{r}{p-r}}\Big)\\
&&\lesssim \psi\Big(c(p,n)[w]_{A_p}^{\frac{1}{p-1}}\Big)[w]_{A_p}^{\max(1,\frac{1}{p-1})}.
\end{eqnarray*}
Applying Lemma \ref{kest} along with the two previous estimates completes the proof.
\end{proof}

\section{Proof of Theorem~\ref{Walshthm}}\label{secwalsh}
Before the actual proof, we  recall the definition of the Walsh-Carleson maximal  operator $\W$. For a nonnegative integer $n$ with dyadic representation $ n= \sum_{j=0}^\infty n_j 2^j,$ we introduce the $n$-th Walsh character by
$$
W_n(x) = \prod_{j=0}^\infty r_j(x)^{n_j}, \qquad x \in \mathbb T\equiv [0,1],$$
where $r_j(x)=\mathrm{sign}\sin (2^j\pi x)$ is the $j$-th Rademacher function. The Walsh characters $\{W_n:n\in \mathbb N\}$ form an orthonormal basis of $L^2(\mathbb T)$. For $f \in L^1(\mathbb T)$, the $n$-th partial Walsh-Fourier sum of $f$ is  \begin{equation} \label{walshseries}
\W_n f(x):= \sum_{k=0}^n \langle f, W_k \rangle W_k (x), \qquad x\in \mathbb{T},
\end{equation}
and  the Walsh-Carleson maximal operator is thus defined by
\begin{equation} \label{walshcarl}
\W f(x)= \sup_{n \in \mathbb N} |\W_n f(x)|, \qquad x\in \mathbb{T}.
\end{equation}

The remainder of this section is devoted to the proof of the inequality below:   for an arbitrary weight   $w$ and $1<p<\infty$,
\begin{equation}
\|\W f\|_{L^p(w)}\lesssim  \inf_{1<r\le 2}\Big\{ r' \sup_{{\mathcal S}}\|{\mathcal A}_{r, \mathcal S}f\|_{L^p(w)}\Big\},
\label{wfinal}
\end{equation}
the supremum being taken over all sparse families of dyadic cubes ${\mathcal S}\subset {\mathcal D}(\mathbb T)$;
in this fully dyadic context, we have redefined
$$
{\mathcal A}_{r, \mathcal S}f(x)=\sum_{Q\in {\mathcal S}}\left(\frac{1}{|  Q|}\int_{  Q}|f|^r\right)^{1/r}\chi_Q(x).
$$
With inequality \eqref{wfinal} in hand, Theorem \ref{Walshthm} follows by essentially repeating the proof given in Sections \ref{sect3} and \ref{sect4} for Theorem \ref{mainr}.  The first step towards \eqref{wfinal} is   a (simpler) substitute of Lemma \ref{oscest}.
\begin{lemma} \label{lemma22walsh}
Let $Q\subset \mathbb T$ be a dyadic interval. Then for all $1<r<2$
\begin{eqnarray}
\o_{\la}(\W f;Q)&\lesssim  r'\left(\frac{1}{|  Q|}\int_{Q}|f|^r\right)^{1/r}.\label{oscw}\end{eqnarray}
\end{lemma}
We postpone the proof of the lemma until after the conclusion of the argument for \eqref{wfinal}. By combining   Theorem \ref{lmoes} applied to $\W f$, with $Q_0=\mathbb T$,   and Lemma \ref{oscest}, we learn   that there exists a sparse family ${\mathcal S}\subset {\mathcal D}(\mathbb T)$ such that for a.e. $x\in \mathbb T$ and $1<r<2$
\begin{equation}\label{intes2}
|\W f(x)-m_{\W f}  (\mathbb T)|\lesssim r' {\mathcal A }_{r, \mathcal S}f(x).
 \end{equation}
Note further that the proof of Theorem \ref{lmoes} naturally yields $Q_0=\mathbb T \in \mathcal S$, so that
\begin{equation}\label{intes3}
|m_{\W f}  (\mathbb T)| \chi_{\mathbb T}(x) \lesssim r' \left(\frac{1}{|  \mathbb T|}\int_{  \mathbb T}|f|^r\right)^{1/r}\chi_{\mathbb T}(x) \leq  r'{\mathcal A}_{r, \mathcal S}f(x)
 \end{equation}
where we applied \eqref{cond1walsh} once again to get the first inequality. The bound \eqref{wfinal} then follows by putting together   \eqref{intes2}-\eqref{intes3}, and arguing as in the proof of Lemma \ref{kest}.

 \begin{proof}[Proof of Lemma \ref{lemma22walsh}]
Here, we will rely on the alternative   representation of $\W f$ as a maximally (Walsh) modulated martingale transform, given by
\begin{align}
\label{modelsum2}
&\W f(x)= \sup_{n \in \mathbb N} \left| T_n f(x)\right|, \\ & \nonumber T_n f(x):=\begin{cases} \W_0 f(x), & n=0, \\ \displaystyle  \sum_{ I \in {\mathcal D}(\mathbb T) } \varepsilon_I \langle (f\cdot W_n), h_I \rangle h_I (x), &n\geq 1,   \end{cases}
\end{align}
 where  $h_I$ is the usual $L^2$-normalized Haar function on $I$, and, for each $n$, $\varepsilon_{I,n} :\mathcal{D}_\mathbb{T}\to \{0,1\}$ is specified below.
The equality   \eqref{modelsum2} is proved at the end of the section; here, we   devote ourselves to \eqref{oscw}. Set $f_1=f\chi_{Q}$ and $f_2=f-f_1$. Let $x\in Q$ and let $x_0$ be the center of $Q$. Then
\begin{eqnarray} \label{ez0}
&&|\W (f)(x)-\W (f_2)(x_0)| =\Big|\sup_{n \in \mathbb N}
|T_n f(x)|-\sup_{n \in \mathbb N }|T_n f_2(x_0)|\Big|\\ \nonumber &&
 \le \sup_{n \in \mathbb N}|T_n f(x)-T_n f_2(x_0)| \\ && \le \W (f_1)(x)+\sup_{n \in \mathbb N}\|T_nf_2(\cdot)-T_n f_2(x_0)\|_{L^{\infty}(Q)}.  \nonumber
\end{eqnarray}
For the local part, by (\ref{cond1walsh}),
$$\big(\W (f_1)\chi_Q\big)^*(\la|Q|)\lesssim r'\left(\textstyle \frac{1}{|  Q|}\int_{  Q}|f|^r\right)^{1/r}.$$
We claim that the second summand in the bottom line of \eqref{ez0} is zero, so that the Lemma follows by combining \eqref{ez0} with the last display.  It is trivial to verify that $T_0f_2(x)=T_0 f_2(x_0)$. Furthermore, for $n \geq 1$,   $$ T_nf_2(x)-T_n f_2(x_0)=
  \sum_{  I \in {\mathcal D}(\mathbb T) } \varepsilon_{I,n} \langle (f_2 W_n),h_I \rangle \big( h_I(x)-h_I(x_0) \big).  $$
and the above sum is also identically zero for $x \in Q$. Indeed, since $f_2$ is supported on $Q^c$ and each $h_I$ is supported on $I$, the sum can be restricted to those $I$ such that $I\cap Q^c \neq \emptyset$. Therefore, either $I \cap Q =\emptyset$, so that both $h_I(x),h_I(x_0)$ vanish; or $I \supsetneq Q$, in which case  $h_I $ is constant on $Q$, and $  h_I(x)-h_I(x_0)=0 $. This finishes the proof of the claim.\end{proof}
\begin{proof}[Proof of the equality \eqref{modelsum2}] Fix $n \geq 1$; let  $\{k_j: j=0,\ldots,J\}$ be the set of those integers $k$  such that $n_k\neq 0$ in the dyadic representation of $n$, in decreasing order with respect to $j$, so that $
n= \sum_{j\leq J} 2^{k_j}.
$
Define
$$
r_0=0, \qquad r_{j }=2^{-k_j}\sum_{\ell < j} 2^{k_\ell}, \quad j=1,\ldots,J ;
$$
it is easy to see that $\{2^{k_j}[r_j,r_{j}+1)\:j=0,\ldots,J-1\}$ are disjoint dyadic intervals which are left children of their dyadic parent and such that the dyadic brother of each contains $n$.
It is shown in \cite[Section 4.1]{ThTR} that
$$
\W_n f (x) = \sum_{j=0}^{J}\sum_{\substack{I \in {\mathcal D}(\mathbb T) \\ |I| = 2^{-k_j }}} \langle f, w_I \rangle w_I, \quad w_I(x)= \frac{\chi_I}{|I|^{\frac12}} W_{r_j}\big(\frac{x}{|I|}\big), \, |I|=2^{-k_j}.
$$
We now use the equality
\begin{equation} \label{pfw2}
 w_I(x)= b_{I} h_{I }(x) W_{n}(x)
\end{equation}
for some $b_I \in \pm 1$ (see \cite[Lemma 2.2]{HL2} for a proof),
to rewrite
$$
\W_n f (x) = W_n(x)\sum_{j=0}^{J}\sum_{{I \in {\mathcal D}(\mathbb T): |I| = 2^{-k_j }}} \langle f W_n,h_I \rangle h_I(x)= W_n(x) T_n f(x) $$
having first noticed  that $b_I$ appears twice, and then having set $\varepsilon_{I,n}=1$ if $|I|=2^{-k_j}$ for some $j$ and zero otherwise in the definition of $T_n$. This shows that $|\W_n f(x)|=|T_n f(x)|$ for each $x\in \mathbb T,n \in \mathbb N$, whence the equality \eqref{modelsum2}.
\end{proof}

\section{Weak $L^p$ from Orlicz type bounds} \label{sectorlicz}
Let $\Phi$ be a Young function, that is, $\Phi:[0,\infty)\to [0,\infty)$, $\Phi$ is continuous,
convex, increasing, $\Phi(0)=0$ and $\Phi(t)\to \infty$ as $t\to \infty$. Define the mean Luxemburg norm of $f$ on a cube $Q\subset {\mathbb R}^n$ by
$$\|f\|_{\Phi,Q}=\inf\left\{\la>0:\frac{1}{|Q|}\int_Q\Phi\left(\frac{|f(x)|}{{\la}}\right)dx\le 1\right\}.$$

\begin{prop}\label{orl} Let $T^{\mathcal F}$ be a maximally modulated Calder\'on-Zygmund operator
satisfying
\begin{equation}\label{cond2}
\|T^{\mathcal F}(f\chi_Q)\|_{L^{1,\infty}(Q)}\lesssim |Q|\|f\|_{\Phi,Q}
\end{equation}
for each cube $Q \subset \mathbb R^n$,
where the Young function $\Phi$ is  such that
$$\g_{\Phi}(p)=\sup_{t\ge 1}\frac{\Phi(t)}{t^{p'}}<\infty$$
for any $p>1$.
Then \eqref{cond1} (and consequently Theorem \ref{mainr}) hold with $\psi(t)=\g_{\Phi}(t)$.
\end{prop}

Before the proof, we apply the Proposition to maximally modulated Calder\'on-Zygmund operators satisfying \begin{equation}\label{restr}
\|T^{\mathcal F}(\chi_E)\|_{L^{p,\infty}}\lesssim (p')^m|E|^{1/p},
\qquad(1<p\le 2)
\end{equation}
for some $m\geq 1$. Exploiting the kernel structure of  $T^{\mathcal F}$ along the lines of Antonov's lemma \cite{A}, it is shown in \cite{GMS} that
  \eqref{cond2} holds  with $\Phi=\Phi_m$ given by $\Phi_m(t)=t(\log({\rm e}+t) )^m\log\log\log ({\rm e}^{{\rm e}^{\rm e}}+t)$. An easy computation then yields
$
\gamma_{\Phi_m}(p)= c_m p^m\log\log ( {{\rm e}^{\rm e}}+p);
$
in particular, recalling that $T^{\mathcal F}=\mathcal C$ falls under the case $m=1$, we obtain \eqref{wbcarl}.

We remark that, for $T^{\mathcal F}=\mathcal C_{lac}$, the best one can take in \eqref{cond2}  \cite{DP}  is $$\Phi(t)=t\log\log({{\rm e}^{\rm e}}+t))\log\log\log\log\big( {\rm e}^{{\rm e}^{{\rm e}^{\rm e}}}+t\big);$$  the above method then yields $\gamma_{\Phi}(p')= c \log( {\rm e}+p')\log\log\log({\rm e}^{{\rm e}^{\rm e}}+p')$, a worse dependence than \eqref{wbcarllac}. To recover \eqref{wbcarllac} via this approach, \eqref{cond2}  with $\Phi(t)=t\log\log({{\rm e}^{\rm e}}+t) $, which is the (sharp) Orlicz function conjectured by Konyagin in \cite{K2} for lacunary Fourier series, is needed.

\subsection*{Proof of Proposition \ref{orl}} We begin with a weak-$L^p$ bound for the Orlicz maximal function
$$
M_\Phi f(x) := \sup_{Q \ni x} \|f\|_{\Phi,Q}.
$$
\begin{lemma}\label{weaklp} For any $p>1$,
\begin{equation}\label{weakmphi}
\|M_{\Phi}f\|_{ {p,\infty}}\lesssim \g_{\Phi}(p')^{1/p}\|f\|_{ p}.
\end{equation}
\end{lemma}
\begin{proof} Let us show that
\begin{equation}\label{pointmphi}
M_{\Phi}f(x)\lesssim \g_{\Phi}(p')^{1/p}M_pf(x).
\end{equation}
Observe that (\ref{weakmphi}) follows immediately from this estimate by the standard weak-type $(p,p)$ property of $M_p$.
Fix a cube $Q$. Set
$$\xi_{\Phi}(p)=\sup_{t\ge \Phi^{-1}(1/2)}\frac{\Phi(t)}{t^p}$$
and
$$\la_0=\Big(2\xi_{\Phi}(p)\frac{1}{|Q|}\int_Q|f|^pdx\Big)^{1/p}.$$
Then
\begin{eqnarray*}
\frac{1}{|Q|}\int_Q\Phi\left(\frac{|f|}{\la_0}\right)dx&=&
\frac{1}{|Q|}\int_{Q\cap\{|f|< \Phi^{-1}(1/2)\la_0\}}\Phi\left(\frac{|f|}{\la_0}\right)dx\\
&+&\frac{1}{|Q|}\int_{Q\cap\{|f|\ge\Phi^{-1}(1/2)\la_0\}}\Phi\left(\frac{|f|}{\la_0}\right)dx\\
&\le& \frac{1}{2}+\xi_{\Phi}(p)
\frac{1}{|Q|}\int_Q\left(\frac{|f|}{\la_0}\right)^pdx=1.
\end{eqnarray*}
Therefore,
$$\|f\|_{\Phi,Q}\le \Big(2\xi_{\Phi}(p)\frac{1}{|Q|}\int_Q|f|^pdx\Big)^{1/p}.$$
Further,
\begin{eqnarray*}
\xi_{\Phi}(p)^{1/p}&=&\max\Big(
\sup_{\Phi^{-1}(1/2)\le t\le 1}\frac{\Phi(t)^{1/p}}{t},\g_{\Phi}(p')^{1/p}\Big)\\
&\le& \max(c(\Phi),\g_{\Phi}(p')^{1/p})\lesssim \g_{\Phi}(p')^{1/p}.
\end{eqnarray*}
Combining this estimate with the previous one proves (\ref{pointmphi}), and completes the proof.
\end{proof}
The main argument for   Proposition \ref{orl} begins now.
Replacing \eqref{cond1} by (\ref{cond2}) and repeating the proof of Lemma \ref{oscest}, we obtain
$$
\o_{\la}(T^{\mathcal F}f;Q)\lesssim \|f\|_{\Phi,\bar Q}+\sum_{m=0}^{\infty}\frac{1}{2^{m\d}}\left(\frac{1}{|2^mQ|}\int_{2^mQ}|f|\right),
$$
which implies
$$M_{\la}^{\#}(T^{\mathcal F}f)(x)\lesssim M_{\Phi}f(x)+Mf(x)\lesssim M_{\Phi}f(x).$$
Hence, combining Lemmata \ref{locshest} and \ref{weaklp}  yields
$$
\|T^{\mathcal F}f\|_{ {p,\infty}}\lesssim \|M_{\la}^{\#}(T^{\mathcal F}f)\|_{L^{p,\infty}}\lesssim \|M_{\Phi}f\|_{ {p,\infty}}
\lesssim \g_{\Phi}(p')^{1/p}\|f\|_{ p}
$$
which is the claim of the Proposition.


\section{Remarks and complements} \label{sect6}
\subsection{On the sharpness of the $A_p$ bounds for $\|{\mathcal C}\|_{L^p(w)}$, $\|{\mathcal C}_{lac}\|_{L^p(w)}$}
 Let $\a_p$ be the best possible exponent in
\begin{equation}\label{bestp}
\|{\mathcal C}\|_{L^p(w)}\lesssim [w]_{A_p}^{\a_p}.
\end{equation}
We read from Corollary \ref{carl} that, up to the $\log\log $ term,
\begin{equation}\label{bestap}
\a_p\le \max\big(p',\frac{2}{p-1}\big).
\end{equation} Further, combining Sj\"olin's bound $\mathcal C: L (\log L)^2 \to L^1$ \cite{SJ} with the converse of Yano's extrapolation theorem of \cite{TaoYano} produces the strong-type estimate \begin{equation}\|{\mathcal C}\|_{L^p}\lesssim p (p')^2, \qquad (1<p<\infty), \label{strongboundlie} \end{equation} which is the best possible asymptotics for $p>2$; up to the $\log\log$ term, this is also obtained  in Corollary \ref{carl} (i).
In heuristic accordance with the ``$L\log L$ conjecture''  mentioned in  \eqref{theconjemb}, it is likely that the dependence in \eqref{strongboundlie}  is  best possible also for $1<p\le2$. Assuming this,
 it is easy to show that the bound \eqref{bestap} for $\a_p$ is sharp, up to the doubly logarithmic term, at least for $1<p\le 2$; incidentally, note that  $\log\log({\rm e}^{\rm e} +[w]_{A_p})$ would drop if one could strengthen the weak-$L^p$ bound \eqref{wbcarl} to
 $$
 \|{\mathcal C } f\|_{ {p,\infty}} \lesssim p'\|f\|_{p}, \qquad (1<p<2),
 $$
which is the same dependence as in \eqref{cond1walsh} for  the model Walsh-Carleson operator $\W$.

Indeed, a well-known argument  by Fefferman-Pipher \cite{FP} (see also \cite{LPR} for an extension) says that if a sublinear operator $T$ satisfies $\|T\|_{L^{p_0}(w)}\lesssim N([w]_{A_1})$ for some $p_0$ and an increasing function $N$, then $\|T\|_{L^r}\lesssim N(cr)$ as $r\to \infty$. Hence, on one hand, since $\|{\mathcal C}\|_{L^r}\simeq r$ as $r\to \infty$,
we obtain that $\a_p\ge 1$ for all $p>1$. On the other hand, let
${\mathcal C}_{\xi(\cdot)}$ be a linearization of ${\mathcal C}$ as in Remark~\ref{rem5}.
Then, by duality and  ~(\ref{bestp}),
$$\|{\mathcal C}_{\xi(\cdot)}^*\|_{L^{p'}(w)}=\|{\mathcal C}_{\xi(\cdot)}\|_{L^{p}(w^{-(p-1)})}\lesssim [w^{-(p-1)}]_{A_{p}}^{\a_p}=[w]_{A_{p'}}^{\a_p(p-1)},$$
and hence $\|{\mathcal C}_{\xi(\cdot)}^*\|_{L^r}\lesssim r^{\a_p(p-1)}$ as $r\to \infty$, which implies
$$\|{\mathcal C}\|_{L^r}\lesssim \frac{1}{(r-1)^{\a_p(p-1)}}$$
as $r\to 1$. Comparing with  \eqref{strongboundlie}, which we have assumed to be sharp, we obtain $\a_p\ge \frac{2}{p-1}$.
Therefore, for all $p>1$,
$$\max\Big(1,\frac{2}{p-1}\Big)\le \a_p\le \max\Big(p',\frac{2}{p-1}\Big).$$
In particular, $\a_p=\frac{2}{p-1}$ for $1<p\le 2$.

Similarly, suppose $\phi_p$ is the best possible function in
$$
\|{\mathcal C}_{lac}\|_{L^p(w)}\lesssim \phi_p([w]_{A_p}).
$$
Then, arguing as above, we obtain that if the unweighted bound for $\|{\mathcal C}_{lac}\|_{L^p}$ in Corollary \ref{carllac} is best possible, then
$$\max(t,t^{\frac{1}{p-1}}\log  ({\rm e}+t))\lesssim \phi(t)\lesssim t^{\max(1,\frac{1}{p-1})}\log  ( {\rm e}+t)\quad(t\ge 1).$$
In particular, $\phi_p(t)=t^{\frac{1}{p-1}}\log ({\rm e}+ t)$ for $1<p\le 2$.

\subsection{On mixed $A_p$-$A_{\infty}$ bounds}
Following recent works, where the $A_p$ bounds were improved by mixed $A_p$-$A_{\infty}$ bounds (see, e.g., \cite{HL,HP,HPR}),
one can obtain similar results for $T^{\mathcal F}$.

Given a weight $w$, define its $A_{\infty}$ constant by
$$[w]_{A_{\infty}}=\sup_{Q}\frac{1}{w(Q)}\int_{Q}M(w\chi_Q)dx.$$
It was shown in \cite{HP} that part (i) of Proposition \ref{sum} holds with the $[w]_{A_1}$ constant replaced by
$[w]_{A_{\infty}}$. Changing only this point in the proof of Theorem \ref{mainr}, part (i), we get
that for any $w\in A_1$ and for all $p>1$,
$$\|T^{\mathcal F}f\|_{L^p(w)}\le c(n,T)pp'\psi(3p')[w]_{A_1}^{\frac{1}{p}}[w]_{A_{\infty}}^{\frac{1}{p'}}\|f\|_{L^p(w)}.$$
For Calder\'on-Zygmund operators this inequality was obtained in \cite{HP}.

Further, it was shown in \cite{HPR} that the property (\ref{bhpr}) holds with $\e=\frac{c(n)}{[\si]_{A_{\infty}}}$,
where, as usual, $\si=w^{-\frac{1}{p-1}}$. Also, observe that the operator ${\mathcal T}_{\mathcal S}$ defined in Section 4 satisfies (see \cite{HL})
$$\|{\mathcal T}_{\mathcal S}\|_{L^p(w)}\lesssim [w]_{A_p}^{\frac{1}{p}}\big([w]_{A_{\infty}}^{\frac{1}{p'}}+[\si]_{A_{\infty}}^{\frac{1}{p}}\big).$$
Changing the corresponding estimates in the proof of Theorem \ref{mainr}, part (ii), we get that for any $p>1$,
$$\|T^{\mathcal F}f\|_{L^p(w)}\le c(n,T^{\mathcal F},p)\psi\Big(c(p,n)[\si]_{A_{\infty}}\Big)[w]_{A_p}^{\frac{1}{p}}\big([w]_{A_{\infty}}^{\frac{1}{p'}}+[\si]_{A_{\infty}}^{\frac{1}{p}}\big).$$

\end{document}